\documentclass[12pt]{elsarticle}

\usepackage{amsmath,amssymb,amsthm}
\usepackage{hyperref}

\newcommand {\R} {\mathbb R}
\newcommand {\N} {\mathbb N}
\newcommand{\bx}{\bar x}

\newcommand{\la}{\lambda}
\newcommand{\de}{\delta}

\newcommand{\ds}{\displaystyle}

\def\lsc{lower semicontinuous}
\def\LHS{left-hand side}
\def\RHS{right-hand side}

\newtheorem{thm}{Theorem}

\newtheorem{cor}[thm]{Corollary}

\newtheorem{prop}[thm]{Proposition}
\theoremstyle{definition}
\newtheorem{defn}[thm]{Definition}
\theoremstyle{remark}
\newtheorem{rem}[thm]{Remark}

\begin{document}
\title{Borwein--Preiss Variational Principle Revisited}
\author[fu]{A. Y. Kruger\corref{cor1}}
\ead{a.kruger@federation.edu.au}
\author[nu]{S. Plubtieng}
\ead{somyotp@nu.ac.th}
\author[nu]{T. Seangwattana}
\ead{seangwattana\_t@hotmail.com}

\cortext[cor1]{Corresponding author}
\address[fu]{Centre for Informatics and Applied Optimization,
Faculty of Science and Technology, Federation University Australia,
Ballarat, Victoria 3353, Australia}
\address[nu]{Department of Mathematics, Faculty of Science, Naresuan University, Phitsanulok 65000, Thailand}
\date{}
\begin{abstract}
In this article, we refine and slightly strengthen the metric space version of the Borwein--Preiss variational principle due to Li, Shi, \emph{J. Math. Anal. Appl.} 246, 308--319 (2000), clarify the assumptions and conclusions of their Theorem~1 as well as Theorem~2.5.2 in Borwein, Zhu, \emph{Techniques of Variational Analysis}, Springer (2005) and streamline the proofs.
Our main result, Theorem~\ref{main} is formulated in the metric space setting.
When reduced to Banach spaces (Corollary~\ref{T4}), it extends and strengthens the smooth variational principle established in Borwein, Preiss, \emph{Trans. Amer. Math. Soc.} 303, 517–-527 (1987) along several directions.
\end{abstract}

\begin{keyword}
Borwein-Preiss variational principle \sep smooth variational principle \sep gauge-type function \sep perturbation
\end{keyword}
\maketitle

\section{Introduction}
The celebrated \emph{Ekeland variational principle} \cite{Eke74} has been around for more than four decades.
It almost immediately became one of the main tools in optimization theory and various branches of analysis.
The number of publications containing ``Ekeland variational principle'' in their title has exceeded 200.
Several other variational principles followed: due to Stegall \cite{Ste78}, Borwein--Preiss \cite{BorPre87}, Deville--Godefroy--Zizler \cite{DevGodZiz93.2} and  others.

Given an ``almost minimal'' point of a function, a variational principle guaranties the existence of another point and a suitably perturbed function for which this point is (strictly) minimal and provides estimates of the (generalized) distance between the points and also the size of the perturbation.
Typically variational principles assume the underlying space to be complete metric (quasi-metric) or Banach and the function (sometimes vector- or set-va\-lued) to possess a kind of semicontinuity.

The principles differ mainly in terms of the class of perturbations they allow.
The perturbation guaranteed by the original Ekeland variational principle (valid in general complete metric spaces) is nonsmooth even if the underlying space is a smooth Banach space and the function is everywhere Fr\'echet differentiable.
In contrast, the \emph{Borwein--Preiss variational principle} (originally formulated in the Banach space setting) works with a special class of perturbations determined by the norm; when the space is \emph{smooth} (i.e., the norm is Fr\'echet differentiable away from the origin), the perturbations are smooth too.
Because of that, the Borwein--Preiss variational principle is referred to in \cite{BorPre87} as the \emph{smooth variational principle}.
It has found numerous applications and paved the way for a number of other smooth principles including the one due to Deville--Godefroy--Zizler \cite{DevGodZiz93.2}.

The statement of the next theorem mostly follows that of \cite[Theorem~2.5.3]{BorZhu05}.

\begin{thm}[Borwein--Preiss variational principle]\label{BP}
Let $(X,\|\cdot\|)$ be a Banach space and function $f : X \rightarrow \mathbb{R}\cup\{+\infty\}$ be \lsc.
Suppose that $\epsilon>0$, $\la>0$ and $p\ge1$.
If $x_{0} \in X$ satisfies
\begin{equation}\label{BP0}
f(x_{0})<\inf_{X} f + \epsilon,
\end{equation}
then there exist a point $\bar x\in X$ and sequences $\{x_{i}\}_{i=1}^\infty\subset X$ and $\{\delta_{i}\}_{i=0}^\infty\subset\R_+\setminus\{0\}$ such that $x_i\to\bx$ as $i\to\infty$, $\sum_{i=0}^{\infty}\delta_{i}=1$, and
\begin{enumerate}
\item
$\|\bar x-x_{i}\|\le\la$ $(i=0,1,\ldots)$;
\item
$\ds f(\bar x) +\frac{\epsilon}{\la^p}\sum\limits_{i=0}^{\infty} \delta_{i}\|\bar x-x_{i}\|^p \le f(x_0)$;
\item
$\ds f(x) +\frac{\epsilon}{\la^p}\sum\limits_{i=0}^{\infty}\delta_{i} \|x-x_{i}\|^p
> f(\bar x) + \frac{\epsilon}{\la^p} \sum\limits_{i=0}^{\infty}\delta_{i}\|\bar x-x_{i}\|^p$
for all
$x \in X\setminus\{\bar x\}$.
\end{enumerate}
\end{thm}

When $X$ is a smooth space and $p>1$, the perturbation functions involved in (ii) and (iii) of the above theorem are smooth.

Among the known extensions of the Borwein--Preiss variational principle, we mention the work by Li and Shi \cite[Theorem 1]{LiShi00}, where the principle was extended to metric spaces (of course at the expense of losing the smoothness) by replacing $\|\cdot\|^p$ in (ii) and (iii) by a more general ``gauge-type'' function $\rho:X\times X\to\R$.
They also strengthened Theorem~\ref{BP} by showing the existence of $\bar x$ and $\{x_{i}\}_{i=1}^\infty$ validating the appropriately adjusted conclusions of the theorem for any sequence $\{\delta_{i}\}_{i=0}^\infty\subset\R_+$ with $\de_0>0$.
This last advancement allowed the authors to cover the Ekeland variational principle which corresponds to setting $\de_i=0$ for $i=1,2,\ldots$
The result by Li and Shi was later adapted in Theorem~2.5.2 in the book by Borwein and Zhu \cite{BorZhu05}.

Another important advancement was made by Loewen and Wang \cite[Theorem 2.2]{LoeWan01} who constructed in the Banach space setting a special class of perturbations subsuming those used in Theorem~\ref{BP} and established strong minimality in the analogue of the condition (iii) above; cf. \cite[Definition~2.1]{LoeWan01}.
Bednarczuk and Zagrodny \cite{BedZag10} extended recently the Borwein--Preiss variational principle to vector-valued functions.

In this article which follows the ideas of \cite{BorPre87, LiShi00,BorZhu05}, we refine and slightly strengthen the metric space version of the Borwein--Preiss variational principle due to Li and Shi \cite{LiShi00}, clarify the assumptions and conclusions of \cite[Theorem 1]{LiShi00} and \cite[Theorem~2.5.2]{BorZhu05} and streamline the proofs.
When reduced to Banach spaces (Corollary~\ref{T4}), our main result extends and strengthens Theorem~\ref{BP} along several directions.

1) The assumption $p\ge1$ for the power index in (ii) and (iii) is relaxed to just $p>0$.
Of course, if $p<1$, then the perturbation function involved in (ii) and (iii) is not convex.

2) The strict inequality \eqref{BP0} is replaced by the corresponding nonstrict one:
\begin{equation}\label{00}
f(x_{0}) \le \inf_{X} f + \epsilon.
\end{equation}
Note that $\de_0$ must satisfy
\begin{gather}\label{ies2}
\delta_{0}\ge(f(x_{0})-\inf_{X} f)/\epsilon
\end{gather}
(see Corollary~\ref{T4}).
Hence, when $f(x_{0})=\inf_{X} f + \epsilon$, one has $\delta_{0}\ge1$ and cannot ensure the equality $\sum_{i=0}^{\infty}\delta_{i}=1$.

3) Instead of assuming the existence of $\bar x$, $\{x_{i}\}_{i=1}^\infty$ and $\{\delta_{i}\}_{i=0}^\infty$ with $\sum_{i=0}^{\infty}\delta_{i}=1$ as in Theorem~\ref{BP}, we show that $\bar x$ and $\{x_{i}\}_{i=1}^\infty$ exist for any sequence $\{\delta_{i}\}_{i=0}^\infty\subset\R_+$ (the fact first observed in \cite{LiShi00}) with $\delta_{0}$ satisfying \eqref{ies2}.
The latter restriction still leaves one enough freedom to choose positive numbers $\de_i$ $(i=1,2,\ldots)$ such that $\sum_{i=0}^{\infty}\delta_{i}<\infty$, thus ensuring the  convergence of the series involved in the \LHS\ of condition (iii).
In the case of the strict inequality \eqref{BP0}, one can obviously satisfy that restriction with some $\delta_{0}<1$ and choose positive numbers $\de_i$ $(i=1,2,\ldots)$ such that $\sum_{i=0}^{\infty}\delta_{i}=1$.

4) Similarly to \cite{LiShi00}, conditions (ii) and (iii) in Theorem~\ref{BP} are complemented by a pair of conditions which correspond to the case when only finitely many elements of the sequence $\de_i$ $(i=0,1,\ldots)$ are nonzero.
These conditions strengthen the corresponding conditions in \cite{LiShi00}.

5) The case when the series $\sum_{i=0}^{\infty}\delta_{i}$ is divergent is not excluded.
We show that the series involved in condition (ii) (and the \RHS\ of condition (iii)) is still convergent.
However, the series in the \LHS\ of condition (iii) can be divergent for some $x\in X$.

6) The inequalities in (i) can be replaced by $\|\bar x-x_{0}\|\le\la$ and $\|\bar x-x_{i}\|\le\epsilon_i$ $(i=1,2,\ldots)$, where $\{\epsilon_{i}\}_{i=1}^\infty$ is an arbitrary sequence of positive numbers.
\medskip

The rest of the article is subdivided into three sections.
In the next one, we present and prove our main result extending the Borwein--Preiss variational principle in metric spaces.
Section~\ref{S3} contains some discussions of the main result and provides several corollaries.
In the final Section~\ref{S4}, we identify developing a ``smooth'' regularity theory as a possible application of the extended Borwein--Preiss variational principle.

Our basic notation is standard, cf. \cite{RocWet98,BorZhu05,DonRoc14}.
$X$ stands for either a metric or a Banach space.
A metric or a norm in $X$ are denoted by $d(\cdot,\cdot)$ or $\|\cdot\|$, respectively.
$\N$ denotes the set of all positive integers.

\section{Extended Borwein--Preiss Variational Principle}\label{S2}

In this section, we extend the metric space version of the Borwein--Preiss variational principle due to Li and Shi \cite{LiShi00} (cf. \cite{BorZhu05}) which subsumes also the Ekeland variational principle.

The theorem below involves sequences indexed by $i\in\N$.
The set of all indices is subdivided into two groups: with $i<N$ and $i\ge N$ where $N$ is an `integer' which is allowed to be infinite: $N\in\N\cup\{+\infty\}$.
If $N=+\infty$, then the first subset of indices is infinite, while the second one is empty.
This trick allows us to treat the cases of a finite and infinite set of indices within the same framework.
Another convention in this section concerns summation over an empty set of indices: $\sum_{k=0}^{-1}a_k=0$.

Following \cite[Theorem 1]{LiShi00} and \cite[Definition 2.5.1]{BorZhu05}, we are going to employ in the rest of the article the following concept of \emph{gauge-type} function.

\begin{defn}\label{D2}
Let $(X,d)$ be a metric space. We say that a continuous function $\rho : X\times X\rightarrow [0,\infty]$ is a gauge-type function if
\begin{enumerate}
\item
$\rho(x,x) = 0$ for all $x \in X,$
\item
for any $\epsilon > 0$ there exists $\delta > 0$ such that, for all $y,z \in X$, inequality $\rho(y,z) \leq \delta$ implies $d(y,z) < \epsilon.$
\end{enumerate}
\end{defn}

Here comes the main result.

\begin{thm}[Extended Borwein--Preiss variational principle]\label{main}
Let $X$ be a complete metric space and a function $f : X \rightarrow \mathbb{R}\cup\{+\infty\}$ be \lsc.
Suppose that $\rho$ is a gauge-type function, $\epsilon>0$, $\{\epsilon_{i}\}_{i=1}^\infty$ and $\{\delta_{i}\}_{i=0}^\infty$ are sequences such that
\begin{itemize}
\item
$\epsilon_{i}>0$ for all $i\in\N$ and $\epsilon_{i}\downarrow0$ as $i\to\infty$;
\item
$\de_{i}>0$ for all $i<N$ and $\de_{i}=0$ for all $i\ge N$, where $N\in\N\cup\{+\infty\}$.
\end{itemize}
If $x_{0} \in X$ satisfies \eqref{00},
then there exist a point $\bar x\in X$ and a sequence $\{x_{i}\}_{i=1}^\infty\subset X$ such that $x_i\to\bx$ as $i\to\infty$ and
\begin{enumerate}
\item
$\ds\rho(\bx,x_{0}) \le \epsilon/\de_0$;
\item
$\ds\rho(\bx,x_{i}) \le \epsilon_i$ $(i=1,2,\ldots)$;
\item
if $N=+\infty$, then the
series $\sum_{i=0}^{\infty} \delta_{i}\rho(\bar x,x_{i})$ is convergent and
\begin{equation}\label{su11}
f(\bar x) + \sum\limits_{i=0}^{\infty} \delta_{i}\rho(\bar x,x_{i}) \le f(x_0);
\end{equation}
otherwise the
series $\sum_{i=N-1}^{\infty}\rho(x_{i+1},x_{i})$ is convergent and
\begin{multline}\label{su11-2}
f(\bx) + \sum_{i=0}^{N-2}\delta_{i} \rho(\bx,x_{i})
\\
+ \delta_{N-1}\sup_{n\ge N-1}\left(\sum_{i=N-1}^{n-1} \rho(x_{i+1},x_{i})+ \rho(\bx,x_{n})\right) \le f(x_{0});
\end{multline}
\item
if $N=+\infty$, then
\begin{equation}\label{su10}
f(x) + \sum\limits_{i=0}^{\infty}\delta_{i}\rho(x,x_{i}) > f(\bar x)+ \sum\limits_{i=0}^{\infty}\delta_{i}\rho(\bar x,x_{i})
\quad\mbox{for all}\quad
x \in X\setminus\{\bar x\};
\end{equation}
otherwise,
for any $x \in X\setminus\{\bar x\}$, there exists an $m_0\ge N$ such that, for all $m\ge m_0$,
\begin{multline}\label{us10}
f(x) + \sum_{i=0}^{N-2}\delta_{i}\rho(x,x_{i}) + \delta_{N-1}\rho(x,x_{m})
> f(\bx)
\\
+ \sum_{i=0}^{N-2}\delta_{i} \rho(\bx,x_{i})
+\delta_{N-1}\sup_{n\ge m} \left(\sum_{i=m}^{n-1}\rho(x_{i+1},x_{i})+ \rho(\bx,x_{n})\right).
\end{multline}
\end{enumerate}
\end{thm}
\begin{proof}
(i) and (ii)
We define sequences $\{x_{i}\}$ and $\{S_{i}\}$ inductively.
Set
\begin{equation}\label{su8}
S_{0} := \{x \in X\mid f(x)+\de_0\rho(x,x_{0}) \le f(x_{0})\}.
\end{equation}
Obviously, $x_{0}\in S_{0}$.
Since the function $x\to f(x)+\de_0\rho(x,x_0)$ is \lsc, subset $S_{0}$ is closed.
For any $x \in S_{0}$, we have
\begin{equation}\label{su1}
\rho(x,x_{0}) \le\frac{f(x_{0}) - f(x)}{\de_0}\le \frac{\epsilon}{\de_0}.
\end{equation}

For $i=0,1,\ldots$, denote $j_i:=\min\{i,N-1\}$, i.e., $j_i$ is the largest integer $j\le i$ such that $\de_j>0$.
Let $i\in\N$ and
suppose $x_{0}, \ldots,x_{i-1}$ and $S_{0},\ldots,S_{i-1}$ have been defined.
We choose $x_{i} \in S_{i-1}$ such that
\begin{equation}\label{su2}
f(x_{i}) + \sum_{k=0}^{j_i-1}\delta_{k}\rho(x_{i},x_{k})
\leq \inf_{x\in S_{i-1}}\left(f(x) + \sum_{k=0}^{j_i-1}\delta_{k}\rho(x,x_{k})\right) + \delta_{j_i}\epsilon_i
\end{equation}
and define
\begin{multline}\label{su4}
S_{i} := \Biggl\{x \in S_{i-1}\mid f(x)+ \delta_{j_i}\rho(x,x_{i})\\
+ \sum_{k=0}^{j_i-1}\delta_{k} (\rho(x,x_{k})-\rho(x_{i},x_{k}))
\leq f(x_{i})\Biggr\}.
\end{multline}
Obviously, $x_{i}\in S_{i}$.
Since the function $x\to f(x)+\delta_{j_i}\rho(x,x_{i})
+ \sum_{k=0}^{j_i-1}\delta_{k} \rho(x,x_{k})$ is \lsc, subset $S_{i}$ is closed.
For any $x \in S_{i}$, we have
$$f(x) - f(x_{i})+\sum_{k=0}^{j_i-1}\delta_{k} (\rho(x,x_{k})-\rho(x_{i},x_{k}))+ \delta_{j_i}\rho(x,x_{i})\le0,$$
and consequently, making use of \eqref{su2},
\begin{multline}\label{su7}
\rho(x,x_{i})\le\frac{1}{\delta_{j_i}} \Biggl(f(x_{i}) + \sum_{k=0}^{j_i-1}\delta_{k}\rho(x_{i},x_{k})
\\
-\Bigl(f(x) + \sum_{k=0}^{j_i-1} \delta_{k}\rho(x,x_{k})\Bigr)\Biggr) \leq\epsilon_i.
\end{multline}
We can see that, for all $i\in\N$, subsets $S_{i}$ are nonempty and closed, $S_{i}\subset S_{i-1}$, and $\sup_{x\in S_i}\rho(x,x_{i}) \rightarrow 0$ as $i\to\infty$.
Since $\rho$ is a gauge-type function, we also have $\sup_{x\in S_i}d(x,x_{i}) \rightarrow 0$ and consequently,
${\rm diam} (S_{i}) \rightarrow 0.$
Since $X$ is complete, $\cap_{i=0}^{\infty}S_{i}$ contains exactly one point; let it be $\bx$.
Hence, $\rho(\bx,x_{i}) \rightarrow 0$ and $x_{i} \rightarrow \bar x$ as $i \rightarrow \infty$.
Thanks to \eqref{su1} and \eqref{su7}, $\bx$ satisfies (i) and (ii).
\medskip

Before proceeding to the proof of claim (iii), we prepare several building blocks which are going to be used when proving claims (iii) and (iv).

Let integers $m$, $n$ and $i$ satisfy $0\le m\le i\le n$.
Since $x_{i+1}\in S_i$ and $\bx\in S_n$, it follows from \eqref{su8} (when $i=0$) and \eqref{su4} that
\begin{gather}\label{su3-}
f(x_{i+1}) + \sum_{k=0}^{j_i-1}\delta_{k} (\rho(x_{i+1},x_{k})-\rho(x_{i},x_{k}))+ \delta_{j_i}\rho(x_{i+1},x_{i}) \leq f(x_{i}),
\\\label{su3-+}
f(\bx) + \sum_{k=0}^{j_n-1}\delta_{k} (\rho(\bx,x_{k})-\rho(x_{n},x_{k}))+ \delta_{j_n}\rho(\bx,x_{n}) \leq f(x_{n}).
\end{gather}
We are going to add together inequalities \eqref{su3-} from $i=m$ to $i=n-1$ and inequality \eqref{su3-+}.
Depending on the value of $N$, three cases are possible.

\underline{If $N>n$}, then $j_i=i$ and $j_n=n$.
Adding inequalities \eqref{su3-} from $i=m$ to $i=n-1$, we obtain
\begin{equation*}
f(x_{n}) + \sum_{k=0}^{n-1}\delta_{k} \rho(x_{n},x_{k}) - \sum_{k=0}^{m-1}\delta_{k} \rho(x_{m},x_{k}) \leq f(x_{m}).
\end{equation*}
Adding the last inequality and inequality \eqref{su3-+}, we arrive at
\begin{equation}\label{su3.3-}
f(\bx) + \sum_{k=0}^{n}\delta_{k} \rho(\bx,x_{k})-\sum_{k=0}^{m-1}\delta_{k} \rho(x_{m},x_{k}) \leq f(x_{m}).
\end{equation}

\underline{If $N\le m$}, then $j_i=N-1$ and $j_n=N-1$.
Adding inequalities \eqref{su3-} from $i=m$ to $i=n-1$, we obtain
\begin{equation*}
f(x_{n}) + \sum_{k=0}^{N-2}\delta_{k} (\rho(x_{n},x_{k})-\rho(x_{m},x_{k}))+ \delta_{N-1}\sum_{k=m}^{n-1}\rho(x_{k+1},x_{k}) \leq f(x_{m}).
\end{equation*}
Adding the last inequality and inequality \eqref{su3-+}, we arrive at
\begin{multline}\label{su3.4-}
f(\bx) + \sum_{k=0}^{N-2}\delta_{k} (\rho(\bx,x_{k})-\rho(x_{m},x_{k})) \\
+ \delta_{N-1}\left(\sum_{k=m}^{n-1}\rho(x_{k+1},x_{k})+ \rho(\bx,x_{n})\right) \leq f(x_{m}).
\end{multline}

\underline{If $m<N\le n$}, we add inequalities \eqref{su3-} separately from $i=m$ to $i=N-1$ and from $i=N$ to $i=n-1$ and obtain, respectively,
\begin{gather*}
f(x_{N}) + \sum_{k=0}^{N-1}\delta_{k} \rho(x_{N},x_{k})-\sum_{k=0}^{m-1}\delta_{k} \rho(x_{m},x_{k}) \leq f(x_{m}),
\\
f(x_{n}) + \sum_{k=0}^{N-2}\delta_{k} (\rho(x_{n},x_{k})-\rho(x_{N},x_{k}))+ \delta_{N-1}\sum_{k=N}^{n-1}\rho(x_{k+1},x_{k}) \leq f(x_{N}).
\end{gather*}
Adding the last two inequalities and inequality  \eqref{su3-+} together, we arrive at
\begin{multline}\label{su3.3-+}
f(\bx) + \sum_{k=0}^{N-2}\delta_{k} \rho(\bx,x_{k})-\sum_{k=0}^{m-1}\delta_{k} \rho(x_{m},x_{k})
\\
+ \delta_{N-1}\left(\sum_{k=N-1}^{n-1}\rho(x_{k+1},x_{k})+ \rho(\bx,x_{n})\right) \leq f(x_{m}).
\end{multline}

(iii) When $N=+\infty$, we set $m=0$ in the inequality \eqref{su3.3-}:
\begin{equation*}
f(\bx) + \sum_{k=0}^{n}\delta_{k}\rho(\bx,x_{k}) \leq f(x_{0}).
\end{equation*}
This inequality must hold for all $n\in\N$.
Hence, the series $\sum_{k=0}^{\infty} \delta_k\rho(\bar x,x_k)$ is convergent and condition \eqref{su11} holds true.

When $N<+\infty$, we set $m=0$ and take $n=N-1$ in the inequality \eqref{su3.3-} and any $n\ge N$ in the inequality \eqref{su3.3-+}:
\begin{gather*}
f(\bx) + \sum_{k=0}^{N-1}\delta_{k} \rho(\bx,x_{k}) \leq f(x_{0}),
\\
f(\bx) + \sum_{k=0}^{N-2}\delta_{k} \rho(\bx,x_{k})+ \delta_{N-1} \left(\sum_{i=N-1}^{n-1}\rho(x_{i+1},x_{i})+ \rho(\bx,x_{n})\right) \leq f(x_{0}).
\end{gather*}
Since $\rho(\bx,x_{n})\to0$ as $n\to\infty$, it follows from the last inequality that the series $\sum_{i=N-1}^{\infty} \rho(x_{i+1},x_{i})$ is convergent.
Combining the two inequalities produces estimate \eqref{su11-2}.

(iv) For any $x \neq \bar x,$ there exists an $m_0\in\N$ such that $x\notin S_{m}$ for all $m\ge m_0$.
By \eqref{su4}, this means that
\begin{equation}\label{su5}
f(x) + \sum_{k=0}^{j_m-1}\delta_{k} (\rho(x,x_{k})-\rho(x_{m},x_{k}))+ \delta_{j_m}\rho(x,x_{m}) > f(x_{m}).
\end{equation}
Depending on the value of $N$, we consider two cases.

\underline{If $N=+\infty$}, then $j_m=m$.
Since the
series $\sum_{k=0}^{\infty} \delta_{k}\rho(\bar x,x_{k})$ is convergent, we can pass in \eqref{su3.3-} to the limit as $n\to\infty$ to obtain
\begin{equation*}
f(\bar x) + \sum_{k=0}^{\infty} \delta_{k}\rho(\bar x,x_{k}) \leq f(x_{m})+ \sum_{k=0}^{m-1}\delta_{k}\rho(x_{m},x_{k}).
\end{equation*}
Subtracting the last inequality from \eqref{su5}, we arrive at \begin{equation*}
f(x) + \sum\limits_{k=0}^{m}\delta_{k}\rho(x,x_{k}) > f(\bar x)+ \sum\limits_{k=0}^{\infty}\delta_{k}\rho(\bar x,x_{k}).
\end{equation*}
Condition (\ref{su10}) follows immediately.

\underline{If $N<\infty$}, we can take $m_0\ge N$.
Then $j_m=N-1$ and it follows from \eqref{su3.4-} that
\begin{multline*}
f(\bx) + \sum_{k=0}^{N-2}\delta_{k} (\rho(\bx,x_{k})-\rho(x_{m},x_{k})) \\
+ \delta_{N-1}\sup_{n\ge m} \left(\sum_{k=m}^{n-1}\rho(x_{k+1},x_{k})+ \rho(\bx,x_{n})\right) \leq f(x_{m}).
\end{multline*}
Subtracting the last inequality from \eqref{su5}, we arrive at \eqref{us10}.
\end{proof}

\section{Comments and Corollaries}\label{S3}
In this section, we discuss the main result proved in Section~\ref{S2} and formulate a series of remarks and several corollaries.

\begin{rem}\label{R18}
1. The series $\sum_{i=0}^{\infty} \delta_{i}\rho(x,x_{i})$ in \eqref{su10} does not have to be convergent for all $x \in X\setminus\{\bar x\}$.

2. If $N<\infty$, in the proof of part (iv) of Theorem~\ref{main} one can also consider the case $m_0<N$.
Then, for $m_0\le m<N$, one has $j_m=m$ and it follows from \eqref{su3.3-+} that
\begin{multline*}
f(\bx) + \sum_{k=0}^{N-2}\delta_{k} \rho(\bx,x_{k})-\sum_{k=0}^{m-1}\delta_{k} \rho(x_{m},x_{k})
\\
+ \delta_{N-1}\sup_{n\ge N} \left(\sum_{k=N-1}^{n-1}\rho(x_{k+1},x_{k})+ \rho(\bx,x_{n})\right) \leq f(x_{m}).
\end{multline*}
Subtracting the last inequality from \eqref{su5}, one arrives at \begin{multline}\label{su10+}
f(x) + \sum_{i=0}^{m}\delta_{i} \rho(x,x_{i}) >
f(\bx)
\\
+ \sum_{i=0}^{N-2}\delta_{i} \rho(\bx,x_{i})+ \delta_{N-1}\sup_{n\ge N} \left(\sum_{k=N-1}^{n-1}\rho(x_{k+1},x_{k})+ \rho(\bx,x_{n})\right).
\end{multline}
This estimate compliments \eqref{us10}.

3. Instead of $\epsilon$-minimality in the sense of \eqref{00}, it is sufficient to assume in Theorem~\ref{main} a weaker form of $\epsilon$-minimality: $f(x)\ge f(x_{0})-\epsilon$ for all $x\in X$ such that $f(x)+\de_0\rho(x,x_{0})>f(x_{0})$.

4. Looking at the statement of Theorem~\ref{main}, it is easy to notice that considering a gauge-type function $\rho$ and a sequence of positive numbers $\{\delta_{i}\}_{i=0}^\infty$ can be replaced by that of a sequence of gauge-type functions $\{\rho_i\}_{i=0}^\infty$ such that, for $i=1,2,\ldots$, function $\rho_i$ is a multiple of $\rho_0$.
The latter assumption can be relaxed or dropped at the expense of weakening or dropping the estimates in part (ii) of the concluding part of Theorem~\ref{main}.

Moreover, one can modify the proof employing in it a sequence of functions $\{\rho_i\}_{i=0}^\infty$ which do not have to possess the second property in Definition~\ref{D2}, as long as they ensure that the resulting sets $S_i$ (cf. \eqref{su4}) are closed and form a decreasing sequence with their diameters going to zero.
This way one can establish additional properties of the sequence $\{x_{i}\}_{i=1}^\infty$ and its limiting point $\bx$.
An interesting example of such a sequence in a Banach space setting was considered by Loewen and Wang \cite{LoeWan01} who proved a strong variant of the Borwein--Preiss variational principle (with $\bx$ being a strong minimizer of the corresponding perturbed function; cf. \cite[Definition~2.1]{LoeWan01}).

5. Setting $\epsilon_i:=\epsilon/(2^i\de_0)$ $(i=1,2,\ldots)$, one can make the estimates in (ii) look as in \cite[Theorem~1]{LiShi00} and \cite[Theorem~2.5.2]{BorZhu05}.

6. Given a positive number $\la$, we can rewrite the conclusion of Theorem~\ref{main} in a more conventional form with $\de_0=1$, $\ds\rho(\bx,x_0)\le \la$ instead of (i) and conditions \eqref{su11} and \eqref{su10} replaced, respectively, with the following ones:
\begin{gather}\tag{\ref{su11}$^\prime$}
f(\bar x)+\frac{\epsilon}{\la}\sum\limits_{i=0}^{\infty} \delta_{i}\rho(\bar x,x_{i}) \le f(x_0),
\\\tag{\ref{su10}$^\prime$}
f(x)+ \frac{\epsilon}{\la} \sum\limits_{i=0}^{\infty}\delta_{i}\rho(x,x_{i}) > f(\bar x)+ \frac{\epsilon}{\la} \sum\limits_{i=0}^{\infty}\delta_{i}\rho(\bar x,x_{i})
\quad\mbox{for all}\quad
x \in X\setminus\{\bar x\}
\end{gather}
and similar amendments in conditions \eqref{su11-2}, \eqref{us10} and \eqref{su10+}.
\end{rem}

The next corollary gives some direct consequences of conditions \eqref{su11-2} and \eqref{us10} in Theorem~\ref{main}.

\begin{cor}\label{C4}
Suppose all the assumptions of Theorem~\ref{main} are satisfied, and $N<\infty$. Then
\begin{gather}\label{C4-11}
f(\bx) + \sum_{i=0}^{N-1}\delta_{i} \rho(\bx,x_{i})
\le f(x_{0}),
\\\label{C4-12}
f(\bx) + \sum_{i=0}^{N-2}\delta_{i} \rho(\bx,x_{i})
+ \delta_{N-1}\sum_{i=N-1}^{\infty} \rho(x_{i+1},x_{i})\le f(x_{0}),
\end{gather}
and, for any $x \in X\setminus\{\bar x\}$, there exists an $m_0\ge N$ such that, for all $m\ge m_0$,
\begin{align}\notag
f(x) + \sum_{i=0}^{N-2}\delta_{i}\rho(x,x_{i}) &+ \delta_{N-1}\rho(x,x_{m})
\\\label{C4-21}
&> f(\bx)
+ \sum_{i=0}^{N-2}\delta_{i} \rho(\bx,x_{i})
+\delta_{N-1} \rho(\bx,x_{m}),
\\\notag
f(x) + \sum_{i=0}^{N-2}\delta_{i}\rho(x,x_{i}) &+ \delta_{N-1}\rho(x,x_{m})
\\\label{C4-22}
&> f(\bx)
+ \sum_{i=0}^{N-2}\delta_{i} \rho(\bx,x_{i})
+\delta_{N-1}\sum_{i=m}^{\infty}\rho(x_{i+1},x_{i}), \end{align}
and consequently,
\begin{align}\notag
f(x) + \sum_{i=0}^{N-2}\delta_{i}\rho(x,x_{i}) &+ \delta_{N-1}\rho(x,\bx)
\\\label{ss10-2}
&\ge f(\bx) + \sum_{i=0}^{N-2}\delta_{i} \rho(\bx,x_{i})
\quad\mbox{for all}\quad
x \in X,
\end{align}
where $\bx$ and $\{x_{i}\}_{i=1}^\infty$ are a point and a sequence guaranteed by Theorem~\ref{main}.
\end{cor}

\begin{proof}
Conditions \eqref{C4-11} and \eqref{C4-12} correspond, respectively, to setting $n=N-1$ and letting $n\to\infty$ under the $\sup$ in condition \eqref{su11-2}.
Similarly, conditions \eqref{C4-21} and \eqref{C4-22} correspond, respectively, to setting $n=m$ and letting $n\to\infty$ under the $\sup$ in condition \eqref{us10}.
Condition \eqref{ss10-2} is obviously true when $x=\bx$.
When $x\ne\bx$, it results from passing to the limit as $m\to\infty$ in any of the conditions \eqref{C4-21} and \eqref{C4-22} thanks to the continuity of $\rho$.
\end{proof}

\begin{rem}
1. Conditions \eqref{C4-11} and \eqref{C4-12} are in general independent.
Conditions \eqref{C4-21} and \eqref{C4-22} are independent too.
Conditions \eqref{C4-11} and \eqref{C4-21} were formulated in \cite{LiShi00}.
Thanks to Corollary~\ref{C4}, Theorem~\ref{main} strengthens \cite[Theorem~1]{LiShi00}.

2. In accordance with Theorem~\ref{main} and Corollary~\ref{C4}, $\bx$ is a point of minimum of the sum $f+g$, where the perturbation function $g$ is defined for $x\in X$ either as
$g(x):=\sum_{i=0}^{\infty}\delta_{i}\rho(x,x_{i})$ if $N=+\infty$ or as $g(x):=\sum_{i=0}^{N-2}\delta_{i}\rho(x,x_{i})+ \delta_{N-1}\rho(x,\bx)$ otherwise.
When $N=+\infty$, the minimum is strict.
Thanks to the next proposition, if function $\rho$ possesses the triangle inequality, the minimum is strict also when $N<+\infty$.
\end{rem}

Recall that a function $\rho:X\times X\to\R$ possesses the \emph{triangle inequality} if $\rho(x_1,x_3)\le \rho(x_1,x_2)+\rho(x_2,x_3)$ for all $x_1,x_2,x_3\in X$.


\begin{prop}\label{C6}
Along with conditions \eqref{C4-11}--\eqref{C4-21}, consider the following one:
\begin{align}\label{C6-4}
f(x) + \sum_{i=0}^{N-2}\delta_{i}\rho(x,x_{i}) &+ \delta_{N-1}\rho(x,\bx)
> f(\bx)
+ \sum_{i=0}^{N-2}\delta_{i} \rho(\bx,x_{i}).
\end{align}
If function $\rho$ possesses the triangle inequality,
then \eqref{C4-12} $\Rightarrow$ \eqref{C4-11} and \eqref{C4-22} $\Rightarrow$ \eqref{C4-21} $\Rightarrow$ \eqref{C6-4}.
\end{prop}

\begin{proof}
For any $m,n\in\N$ with $m<n$, we have
$$
\rho(\bx,x_{m})\le\rho(\bx,x_{n})+ \sum_{i=m}^{n-1} \rho(x_{i+1},x_{i}),
$$
and consequently, passing to the limit as $n\to\infty$,
\begin{align*}
\rho(\bx,x_{m})\le\sum_{i=m}^{\infty} \rho(x_{i+1},x_{i}).
\end{align*}
Hence,  \eqref{C4-12} $\Rightarrow$ \eqref{C4-11} and \eqref{C4-22} $\Rightarrow$ \eqref{C4-21}.
Condition \eqref{C6-4} follows from \eqref{C4-21} thanks to the inequality $\rho(x,x_{m})\le\rho(x,\bx)+\rho(\bx,x_{m})$.
\end{proof}

\begin{cor}\label{C6-0}
Suppose all the assumptions of Theorem~\ref{main} are satisfied, $N<+\infty$, and function $\rho$ possesses the triangle inequality.
Then condition \eqref{C6-4} holds true for all $x\in X\setminus\{\bx\}$.
\end{cor}

\begin{proof}
The statement is a consequence of Corollary~\ref{C4} thanks to Proposition~\ref{C6}.
\end{proof}

The next two statements are consequences of Theorem~\ref{main} when $N=+\infty$ and $N=1$, respectively, and $\rho$ is of a special form.
The first one corresponds to the case $N=+\infty$, $X$ a Banach space and $\rho(x_1,x_2):= \|x_1-x_2\|^p$ where $p>0$.

\begin{cor}\label{T4}
Let $(X,\|\cdot\|)$ be a Banach space and function $f : X \rightarrow \mathbb{R}\cup\{+\infty\}$ be \lsc.
Suppose that $\la$, $p$, $\epsilon$, $\epsilon_{i}$ $(i=1,2,\ldots)$, $\delta_{i}$ $(i=0,1,\ldots)$ are positive numbers and $\epsilon_{i}\downarrow0$ as $i\to\infty$.
If $x_{0} \in X$ and $\delta_{0}$ satisfy conditions \eqref{00} and \eqref{ies2},
then there exist a point $\bar x\in X$ and a sequence $\{x_{i}\}_{i=1}^\infty\subset X$ such that $x_i\to\bx$ as $i\to\infty$ and
\begin{enumerate}
\item
$\|\bar x-x_{0}\|\le\la$;
\item
$\|\bar x-x_{i}\|\le\epsilon_i$ $(i=1,2,\ldots)$;
\item
$\ds f(\bar x) +\frac{\epsilon}{\la^p}\sum\limits_{i=0}^{\infty} \delta_{i}\|\bar x-x_{i}\|^p \le f(x_0)$;
\item
$\ds f(x) +\frac{\epsilon}{\la^p}\sum\limits_{i=0}^{\infty}\delta_{i} \|x-x_{i}\|^p
> f(\bar x) + \frac{\epsilon}{\la^p} \sum\limits_{i=0}^{\infty}\delta_{i}\|\bar x-x_{i}\|^p$
for all
$x \in X\setminus\{\bar x\}$.
\end{enumerate}
\end{cor}
\begin{proof}
Set $\rho(x_1,x_2):= \|x_1-x_2\|^p$, $x_1,x_2\in X$.
It is easy to check that $\rho$ is a gauge-type function.
Set $\epsilon':=\epsilon\de_0$, $\epsilon_i':=\epsilon_i^{p}$ $(i=1,2,\ldots)$, $\de_i':=(\epsilon/\la^p)\de_i$ $(i=0,1,\ldots)$.
Then $f(x_{0})\le\inf_{X}f+\epsilon'$, $\epsilon_i'\downarrow0$ as $i\to\infty$ and $\epsilon'/\de_0'=\la^p$.
The conclusion follows from Theorem~\ref{main} with $\epsilon'$, $\epsilon_i'$ and $\de_i'$ in place of $\epsilon$, $\epsilon_i$ and $\de_i$, respectively.
\end{proof}
Condition (iv) means that $\bx$ is a point of strict minimum of the function $x\mapsto f(x)+(\epsilon/\la^p)g(x)$, where $g(x):=\sum_{i=0}^{\infty}\delta_{i} \|x-x_{i}\|^p$.
If $X$ is Fr\'echet smooth, $p>1$, and $\sum_{i=0}^{\infty}\delta_{i}<\infty$, then $g$ is defined on the whole of $X$ and is everywhere Fr\'echet differentiable, i.e., we have an example of a smooth variational principle of Borwein--Preiss type.

\begin{rem}\label{R11}
1. Apart from \eqref{ies2}, no other restrictions are imposed on the positive numbers $\delta_{i}$, $i=0,1,\ldots$

2. Condition \eqref{00} does not exclude the equality case: $f(x_{0})=\inf_{X}f+\epsilon$.
In the latter case, condition \eqref{ies2} is equivalent to $\de_0\ge1$.
This still allows one to chose positive numbers $\delta_{i}$, $i=1,2,\ldots$, such that $\sum_{i=0}^{\infty}\delta_{i}<\infty$ if necessary.

When the inequality \eqref{00} is strict, then one can choose $\de_0<1$ and positive numbers $\delta_{i}$, $i=1,2,\ldots$, such that $\sum_{i=0}^{\infty}\delta_{i}=1$.
\end{rem}

The next statement is the Ekeland variational principle.
It corresponds to $N=1$ and $\rho$ being a distance function.

\begin{cor}\label{T3}
Let $(X,d)$ be a complete metric space and function $f : X \rightarrow \mathbb{R}\cup\{+\infty\}$ be \lsc.
Suppose $\la>0$ and $\epsilon>0$.
If $x_{0} \in X$ satisfies \eqref{00}, then
there exists a point $\bar x\in X$ such that
\begin{enumerate}
\item
$d(\bx,x_{0})\le \la$;
\item
$\ds f(\bx) +\frac{\epsilon}{\la}d(\bx,x_{0}) \le f(x_{0});$
\item
$\ds f(x) + \frac{\epsilon}{\la}d(x,\bx) > f(\bx)$
for all $x \in X\setminus\{\bar x\}$.
\end{enumerate}
\end{cor}
\begin{proof}
Set $\rho:=d$, $N=1$,
$\de_0:=\epsilon/\la$,
$\epsilon_i:=\epsilon/2^{i}$ and $\de_i:=0$ $(i=1,2,\ldots)$.
Then $\epsilon_i\downarrow0$ as $i\to\infty$ and $\epsilon/\de_0=\la$.
The conclusion follows from Theorem~\ref{main} and Corollary~\ref{C6-0}.
\end{proof}

\section{``Smooth'' Regularity Theory}\label{S4}
One can try to use the estimates in Theorem~\ref{main} for developing a ``smooth'' regularity theory similar to the conventional theory based on the application of the Ekeland variational principle (cf. \cite{Aze06,Iof00_,DonRoc14}) and usually using certain \emph{slopes} to formulate primal space criteria (cf. \cite{Aze06,Iof00_,Kru15,Kru15.2}).
The first step towards the development of such a theory would be defining appropriate ``smooth'' slopes.

To illustrate the idea, we consider briefly the case $N=+\infty$.
Let a function $f : X \rightarrow \mathbb{R}\cup\{+\infty\}$, a gauge-type function $\rho : X\times X\rightarrow [0,\infty]$ and a sequence $\{\delta_{i}\}_{i=0}^\infty\subset\R_+\setminus\{0\}$ with $\de_0=1$ be given.

For a sequence $\{x_{i}\}_{i=0}^\infty\subset X$ , define
$$g_{\{x_{i}\}}(u):=\sum_{i=0}^{\infty} \delta_{i}\rho(u,x_{i}),
\quad u\in X.$$
Next, for an $x\in X$ with $f(x)<\infty$ and a sequence $\{x_{i}\}_{i=0}^\infty\subset X$ convergent to $x$ with $g_{\{x_{i}\}}(x)<\infty$, the slope of $f$ at $(x,\{x_{i}\})$ can be defined as follows:
\begin{equation}\label{li}
|\nabla{f}|(x,\{x_{i}\}):= \limsup_{\substack{u\to{x}\\g_{\{x_{i}\}}(u)\ne g_{\{x_{i}\}}(x)}}\;
\frac{[f(x)-f(u)]_+}{g_{\{x_{i}\}}(u)-g_{\{x_{i}\}}(x)}.
\end{equation}
Similarly to the conventional slope, this quantity characterizes the maximal `rate of descent' of $f$ at $x$ (with respect to $g_{\{x_{i}\}}$).

Theorem~\ref{main} implies the existence of a point $\bar x\in X$ near the given point $x_0$ and a sequence $\{x_{i}\}_{i=1}^\infty\subset X$ convergent to $\bx$ such that $|\nabla{f}|(\bx,\{x_{i}\})$ is small.
Moreover, it provides quantitative estimates for $|\nabla{f}|(\bx,\{x_{i}\})$ and the `distance' (in terms of $\rho$) from $\bx$ to $x_0$.
More specifically, in the framework of Remark~\ref{R18}.6, one has $\ds\rho(\bx,x_0)\le \la$ and $|\nabla{f}|(\bx,\{x_{i}\})\le\epsilon/\la$.

Furthermore, since (\ref{su10}) (and (\ref{su10}$^\prime$)) is a global condition, it could make sense to incorporate along with the slope \eqref{li} a \emph{nonlocal} analogue of \eqref{li} as well as their \emph{strict} (outer) extensions along the lines of \cite{Kru15,Kru15.2}.
In the Banach space setting and with $\rho$ appropriately defined (cf. Corollary~\ref{T4}), one can try to define a dual space counterpart of \eqref{li} and formulate subdifferential consequences of Theorem~\ref{main} exploiting the original idea of Borwein and Preiss \cite{BorPre87}.

This type of conditions should be useful when developing ``smooth'' criteria of error bounds and metric (H\"older) (sub-)regularity along the lines of \cite{Kru15,Kru15.2}.

The case $N<\infty$ is also of interest and can be handled in a similar way.
The appropriate definitions of slopes can be derived from condition \eqref{us10} (or its `$m$-free' consequence \eqref{ss10-2}).

This topic goes beyond the scope of the current article and is left for future research.
Extending Theorem~\ref{main} and its corollaries to vector-valued functions seems to be another interesting direction of future research.

\section*{Acknowledgments}
The research was supported by the Australian Research Council, project DP110102011; Naresuan University, and Thailand Research Fund, the Royal Golden Jubilee Ph.D. Program.

\section*{References}

\bibliographystyle{elsarticle-num}
\bibliography{buch-kr,kruger,kr-tmp}
\end{document}